\documentclass{article}

\usepackage{amsfonts}
\usepackage{amsmath}
\usepackage{amssymb}
\usepackage{amsthm}
\usepackage{xypic}
\xyoption{curve}
\usepackage{graphicx}

\newcommand{\etal}{{et al.\! }}

\newcommand{\ie}{{i.e., }}

\renewcommand{\S}[1]{\ensuremath{\mathfrak{S}_{#1}}} 

\newcommand{\Sav}[2]{\ensuremath{\mathfrak{S}_{#1}(#2)}} 

\newcommand{\sav}[2]{\ensuremath{\bigl| \mathfrak{S}_{#1}(#2) \bigr|}} 

\newcommand{\Sp}[3]{\ensuremath{\mathfrak{S}_{#1}(#2)[#3]}}
\renewcommand{\sp}[3]{\ensuremath{s_{#1}(#2)[#3]}}
\newcommand{\Spt}[4]{\ensuremath{\mathfrak{S}_{#1}(#2)[#3;#4]}}
\newcommand{\spt}[4]{\ensuremath{s_{#1}(#2)[#3;#4]}}

   \newcommand{\oi}{\sim}
   
   \newcommand{\statwe}[1]{\ensuremath{\stackrel{#1}{\equiv}}}  

\newcommand{\inv}{\ensuremath{\textrm{inv}}}
\newcommand{\maj}{\ensuremath{\textrm{maj}}}
\newcommand{\rmaj}{\ensuremath{\textrm{rmaj}}}
\newcommand{\des}{\ensuremath{\mathrm{des}}}
\newcommand{\peak}{\ensuremath{\mathrm{peak}}}
\newcommand{\vall}{\ensuremath{\mathrm{vall}}}
\newcommand{\sgn}{\ensuremath{\mathrm{sgn}}}
\newcommand{\red}{\ensuremath{\mathrm{red}}}
\newcommand{\ltrmin}{\ensuremath{\mathrm{ltrmin}}}
\newcommand{\ltrmax}{\ensuremath{\mathrm{ltrmax}}}
\newcommand{\rtlmin}{\ensuremath{\mathrm{rtlmin}}}
\newcommand{\rtlmax}{\ensuremath{\mathrm{rtlmax}}}
\newcommand{\rtl}{right-to-left }
\newcommand{\ltr}{left-to-right }

\newtheorem{theorem}{Theorem}
\newtheorem{definition}[theorem]{Definition}

\newtheorem{lemma}[theorem]{Lemma}
\newtheorem{corollary}[theorem]{Corollary}

\newtheorem*{remark}{Remark}

\renewcommand{\d}{\mbox{-}}

\begin{document}
\title{Refining enumeration schemes to count according to permutation statistics}
\author{Andrew M. Baxter}


\maketitle

\abstract{
We consider the question of computing the distribution of a permutation statistics over restricted permutations via enumeration schemes.  The restricted permutations are those avoiding sets of vincular patterns (which include both classical and consecutive patterns), and the statistics are described in the number of copies of certain vincular patterns such as the descent statistic and major index.  An enumeration scheme is a polynomial-time algorithm (specifically, a system of recurrence relations) to compute the number of permutations avoiding a given set of vincular patterns.  Enumeration schemes' most notable feature is that they may be discovered and proven via only finite computation.  We prove that when a finite enumeration scheme exists to compute the number of permutations avoiding a given set of vincular patterns, the scheme can also compute the distribution of certain permutation statistics with very little extra computation.
}

\section{Introduction}

Enumeration schemes are special recurrences to compute the number of permutations avoiding a set of vincular patterns.  In this paper, we discuss how to refine enumeration schemes to compute the distributions of certain permutation statistics over a set of pattern-avoiding permutations.  This extends previous work in \cite{Baxter2010, Baxter2012} which considers the distribution of only the inversion number.

Let $[n]$ be shorthand for the set $\{1,\dotsc,n\}$.  For a word $w \in [n]^k$, we write  $w=w_1 w_2 \dotsm w_k$ and define the \emph{reduction} $\red(w)$ to be the word obtained by replacing the $i^{th}$ smallest letter(s) of $w$ with $i$.  For example $\red(839183)=324132$.  If $\red(w)=\red(w')$, we say that $w$ and $w'$ are \emph{order-isomorphic} and write $w\oi w'$.  We will commonly use the notation $|w|$ to denote the length of $w$.
 

   Vincular patterns resemble classical patterns, with the constraint that some of the letters in a copy must be consecutive.  Formally, a \emph{vincular pattern of length $k$} is a pair $(\sigma, X)$ where $\sigma$ is a permutation in $\S{k}$ and $X\subseteq \{0,1,2,\dotsc, k\}$ is a set of ``adjacencies.''   A permutation $\pi\in\S{n}$ \emph{contains the vincular pattern $(\sigma, X)$} if there is a $k$-tuple $1\leq i_1 < i_2 < \dotsb < i_k\leq n$ such that the following three criteria are satisfied:    
   \begin{itemize}
     \item $\red(\pi_{i_1} \pi_{i_2} \dotsm \pi_{i_k}) = \sigma$.
     \item $i_{x+1} = i_x +1$ for each $x\in X\setminus\{0,k\}$.
     \item $i_1=1$ if $0\in X$ and $i_k=n$ if $k\in X$.
   \end{itemize}
In the present work we restrict our attention to patterns $(\sigma, X)$ where $\sigma\in\S{k}$ and $X\subseteq [k-1]$, rendering the third containment criterion irrelevant.\footnote{We enact this restriction partly for simplicity.  It is likely that the prefix-focused arguments in \cite{Baxter2012} and below extend to patterns $(\sigma, X)$ with $0\in X$ with few modifications, but it is unlikely such an approach could work for patterns with $k \in X$.}
The subsequence $\pi_{i_1} \pi_{i_2} \dotsm \pi_{i_k}$ is called a \emph{copy} of $(\sigma, X)$.  In the permutation $\pi= 162534$, the subsequence $1253$ is a copy of $(1243, \{3\})$, but the subsequence $1254$ is not a copy since the 5 and 4 are not adjacent in $\pi$.  The ``classical pattern'' $\sigma$ is precisely the vincular pattern $(\sigma, \emptyset)$ since no adjacencies are required, while the ``consecutive pattern'' $\sigma$ is the vincular pattern $(\sigma, \{1,2,\dotsc, k-1\})$ since all internal adjacencies are required.  
 
   In practice we write $(\sigma, X)$ as a permutation with a dash between $\sigma_j$ and $\sigma_{j+1}$ if $j\not\in X$.  For example, $(1243, \{3\})$ is written $1\d2\d43$.   We occasionally refer to ``the vincular pattern $\sigma$'' or even ``the pattern $\sigma$'' without explicitly referring to $X$.  
   
   If the permutation $\pi$ does not contain a copy of the pattern $(\sigma, X)$, then $\pi$ is said to \emph{avoid} $\sigma$.  We will notation $\Sav{n}{\sigma}$ or $\Sav{n}{(\sigma, X)}$ to denote the set of permutations avoiding the $(\sigma, X)$, and $\Sav{n}{B}$ denotes those permutations avoiding every vincular pattern $(\sigma, X) \in B$.

  Observe that a vincular pattern $(\sigma, X)$ of length $k$ exhibits similar symmetries to those of permutations, except for taking inverses.  The reverse is given by $(\sigma,X)^{r} = (\sigma^{r}, k-X)$ where $k-X = \{k-x: x\in X\}$.  For example, $(1\d3\d42)^{r} = 24\d3\d1$.  The complement is $(\sigma,X)^{c}=(\sigma^{c},X)$.  For example, $(1\d3\d42)^{c} = 4\d2\d13$.  It follows that that $\pi$ avoids $(\sigma, X)$ if and only if $\pi^{r}$ avoids $(\sigma, X)^{r}$.  Similarly, $\pi$ avoids $(\sigma, X)$ if and only if $\pi^{c}$ avoids $(\sigma, X)^{c}$.
    

See Steingr\'{i}mssson's survey for a fuller history in \cite{Steingrimsson2010Survey}.  From their earliest days, vincular patterns been linked to many of the common combinatorial structures such as set partitions and lattice paths in \cite{Claesson2001} as well as permutation statistics in \cite{Babson2000}.
    
  Enumeration schemes were introduced by Zeilberger in \cite{Zeilberger1998} as an automated method to compute $\sav{n}{B}$ for many different $B$.  Vatter improved schemes in \cite{Vatter2008} with the introduction of gap vectors, and Zeilberger provided an alternate implementation in \cite{Zeilberger2006}.  The greatest feature of schemes is that they may be discovered by a computer: the user need only input the set $B$ (along with bounds to the computer search) and the computer will return an enumeration scheme (if one exists within the bounds of the search) which computes $\sav{n}{B}$ in polynomial time. Pudwell extended these methods to consider pattern avoidance in permutations of a multiset in \cite{Pudwell2010a, Pudwell2008}, as well as barred-pattern avoidance in \cite{Pudwell2010b}.  The author and Pudwell extended schemes to sets of vincular patterns in \cite{Baxter2012}.

 A \emph{permutation statistic} is any function $f: \bigcup_{n\geq 0} \S{n} \to \mathbb{Z}$.  The most-studied permutation statistic is the \emph{inversion number} $\inv(\pi) = \bigl| \{(i,j): i<j \text{ and } \pi_i > \pi_j \} \bigr|$.  In terms of vincular patterns, $\inv(\pi)$ is the number of copies of $2\d1 = (21, \emptyset)$.  Similarly, the \emph{descent number} $\des(\pi) = \bigl| \{i: \pi_i>\pi_{i+1} \} \bigr|$ is the number of copies of $21 = (21, \{1\})$.  In this work we will primarily consider permutation statistics which count the number of copies of a given vincular pattern, sometimes called \emph{pattern functions}.  It is shown in \cite{Babson2000} that many well-known permutation statistics can be framed as linear combinations of pattern functions.  For a permutation statistic $f$ and set $S\subseteq \bigcup_{n\geq 0} \S{n}$, the \emph{distribution of $f$ over $S$} is given  by:
 \begin{equation}
   F(S, f, q) := \sum_{\pi\in S} q^{f(\pi)}
 \end{equation}
We also consider the simultaneous distribution of multiple statistics $\mathbf{f} = \langle f_1, \dotsc, f_m \rangle$ over the same set $S$ with the indeterminates $\mathbf{q} = \langle q_1 \dotsc, q_m\rangle$.  The distribution of $\mathbf{f}$ over $S$ is given by:
 \begin{equation}
   F(S,  \langle f_1, \dotsc, f_m \rangle,  \langle q_1, \dotsc, q_m \rangle) := \sum_{\pi\in S} {q_1}^{f_1(\pi)}\,{q_2}^{f_2(\pi)}\dotsm {q_m}^{f_m(\pi)}.
 \end{equation}

Distributions of statistics over sets of pattern-avoiding permutations have received increased attention of late, focusing primarily on sets of permutations avoiding classical patterns of length 3.   Barcucci \etal use generating trees to study the inversion number over $\Sav{n}{B}$ for a few examples of sets $B$ in \cite{Barcucci2001}.  Barnabei \etal study copies of consecutive patterns over $\Sav{n}{1\d2\d3}$ and $\Sav{n}{3\d1\d2}$  in \cite{Barnabei2009, Barnabei2010}.  Dokos \etal refine Wilf-equivalence in \cite{Dokos2012} by studying the inversion number and major index over $\Sav{n}{\tau}$ for classical patterns $\tau\in\S{3}$.  Bona and Homberger study the total number of classical patterns $\sigma\in\S{3}$ over $\Sav{n}{\tau}$ for another classical pattern $\tau\in\S{3}$ in \cite{Bona2010, Bona2012, Homberger2012}.  Most recently, Burstein and Elizalde in \cite{Burstein2013} study the total number of vincular patterns of length $3$ over $\Sav{n}{\tau}$ for classical patterns $\tau\in\S{3}$.

Suppose that $E$ is a finite enumeration scheme which gives a recurrence to compute $\sav{n}{B}$ for a given set of patterns $B$.  The work in \cite{Baxter2010, Baxter2012} demonstrates how to use $E$ to compute $F(\Sav{n}{B}, \inv, q)$. The present work demonstrates how to use $E$ to compute the distribution $F(\Sav{n}{B},\mathbf{f}, \mathbf{q})$ where each statistic $f_i$ counts the number of copies of a vincular pattern of the form $(\sigma, [ |\sigma| -1)$ or $(\sigma, [|\sigma|-2])$ or counts the number of \rtl minima or \rtl maxima.  The results are implemented in the Maple package \texttt{Statter}, available for download from the author's homepage.

The paper is organized as follows.  Section \ref{sec:schemesoverview} outlines the basics of enumeration schemes and their structure.  Section \ref{sec:deletionfriend} defines the notion of an ``enumeration-scheme-compatible,'' or ``ES-compatible,'' statistic.  Subsection \ref{sec:deletionfriendlyexamples} presents three classes of ES-compatible statistics.  Section \ref{sec:Deepening} presents a technical result proving that any given finite enumeration scheme can be expanded to fit the additional requirements which ES-compatible statistics can impose.   Section \ref{sec:Applications} presents three specific examples of how enumeration schemes can be applied to explore statistics over sets $\Sav{n}{B}$.

\section{Overview of Enumeration Schemes}\label{sec:schemesoverview}

  Enumeration schemes are succinct encodings for a family of recurrence relations enumerating a family of sets.  The enumerated sets are actually subsets of $\Sav{n}{B}$ determined by prefixes.  
  
  For pattern $p\in \S{k}$, let $\Sp{n}{B}{p}$ be the set of permutations $\pi\in \Sav{n}{B}$ such that $\red(\pi_1 \pi_2 \dotsc \pi_k)= p$.  We call $p$ the \emph{prefix pattern}.  To refine further, let $w\in[n]^k$ and define $\Spt{n}{B}{p}{w}$ to be those permutations in $\pi\in \Sp{n}{B}{p}$ such that $\pi_1 \pi_2 \dotsm \pi_k=w$.  For example,
\begin{equation*}
 \Spt{5}{1\d2\d3}{21}{53} = \{ 53142, 53214, 53241, 53412, 53421 \}.
\end{equation*}
Since we are interested in enumeration, it will be handy to have the notation $\sp{n}{B}{p}=\bigl|\Sp{n}{B}{p}\bigr|$ and $\spt{n}{B}{p}{w}=\bigl|\Spt{n}{B}{p}{w}\bigr|$.

By looking at the prefix of a permutation, one can identify likely ``trouble spots'' where forbidden patterns may appear.  For example, suppose we wish to avoid the (classical) pattern $1\d2\d3$.  Then the presence of the pattern $12$ in the prefix indicates the potential for the whole permutation to contain a $1\d2\d3$ pattern.  

Enumeration schemes take a divide-and-conquer approach to enumeration.  We define the \emph{child} of a permutation $p\in\S{k}$ to be any permutation $p'\in\S{k+1}$ such that $\red(p'_1 p'_2 \dotsm p'_k)=p$.
Any $\Sp{n}{B}{p}$ for $p\in\S{k}$ may be partitioned into the family of sets $\Sp{n}{B}{p'}$ for each of its children $p' \in \Sp{k+1}{B}{p}$.  The sets indexed by these children are then considered as described below, and their sizes are totaled to obtain $\sp{n}{B}{p}$.  In the end we have computed $\sav{n}{B}$, since $\Sav{n}{B} = \Sp{n}{B}{\epsilon} = \Sp{n}{B}{1}$, where $\epsilon$ is the empty (\ie, length 0) permutation.

For $p\in S_k$ a set $\Sp{n}{B}{p}$ fits into one of three cases:

\begin{enumerate}
   \item[(1)] If $n=k$, then $\Sp{n}{B}{p}$ is either $\{p\}$ or $\emptyset$, depending on whether $p$ avoids $B$.  
   \item[(2)] For each $w\in[n]^k$ such that $\red(w)=p$, one of the following happens:
     \begin{itemize}
      \item[(2a)] $\Spt{n}{B}{p}{w}$ is empty, and so $\spt{n}{B}{p}{w}=0$.
      \item[(2b)] $\Spt{n}{B}{p}{w}$ is in bijection with some other $\Spt{\hat{n}}{B}{\hat{p}}{\hat{w}}$ for $\hat{n}<n$, and so $\spt{n}{B}{p}{w} = \spt{\hat{n}}{B}{\hat{p}}{\hat{w}}$.
     \end{itemize}
   \item[(3)] $\Sp{n}{B}{p}$ must be partitioned further, so $\sp{n}{B}{p} = \sum\limits_{p'\in \Sp{k+1}{B}{p}} \sp{n}{B}{p'}$.
\end{enumerate}

Case (1) provides the base cases for our recurrence.  If case (2) applies, then we will use it preferentially over case (3).
If case (2) does not apply, we must divide $\Sp{n}{B}{p}$ as in case (3).  Determining whether case (2) applies makes use of \emph{gap vector criteria} to test (2a) and \emph{reversible deletions} to form the bijection in (2b).  These concepts are outlined in the following subsections.

\subsection{Gap Vectors}\label{gv:intro}

The differences between the values of letters in the prefix may be great that a forbidden pattern \emph{must} appear.   To make this more precise, we follow our example above and compute $\sp{n}{1\d2\d3}{12}$.  Observe that $\Spt{n}{1\d2\d3}{12}{w_1 w_2}$ is empty if $w_1<w_2<n$, since otherwise if $\pi \in \Spt{n}{1\d2\d3}{12}{w_1 w_2}$ then $\pi_i=n$ for some $i\geq 3$ and so $w_1 w_2 n$ forms a $1\d2\d3$ pattern.  Since the possibility for any $\pi_i>w_2$ for $i\geq 3$ prohibits the formation of a $1\d2\d3$-avoiding permutation, we must restrict the space above $w_2$.

To formalize this, consider $\Spt{n}{B}{p}{w}$ and let $c_i$ be the $i^{th}$ smallest letter in $w$.  Let $c_0=0$ and $c_{k+1}=n+1$, and form the $(k+1)$-vector $\vec{g}(n,w)$ so that the $i^{th}$ component is $g_i = c_i-c_{i-1}-1$.  Note that $g_i$ counts the number of letters for any $\pi\in \Spt{n}{B}{p}{w}$ which lie strictly between $c_{i-1}$ and $c_{i}$, \ie the number of letters $\pi_j$ following the prefix ($j>k$) and $c_{i-1} \leq \pi_j \leq c_{i}$.  We call $\vec{g}(n,w)$ the \emph{spacing vector} for $w$.

In the example above, if $\vec{g}(n,w)\geq \langle 0,0,1\rangle$ in the product order of $\mathbb{N}^{3}$ (\ie component-wise), then $\Spt{n}{1\d2\d3}{12}{w}=\emptyset$.  We call $\langle 0,0,1 \rangle$ a gap vector for the prefix $12$.  More generally we may make the following definition:

\begin{definition}
Given a set of forbidden patterns $B$ and prefix $p$, then $\vec{v}$ is a \emph{gap vector for prefix $p$ with respect to $B$} if, for all $n$, $\Spt{n}{B}{p}{w}=\emptyset$ for any $w$ such that $\vec{g}(n,w)\geq \vec{v}$.  In this case we say that $w$ \emph{satisfies} the \emph{gap vector criterion} for $\vec{v}$.  
\end{definition}

Hence $\vec{v}=\langle 0,0,1 \rangle$ is a gap vector for $p=12$ with respect to $B=\{1\d2\d3\}$, and any prefix set $w = w_1 w_2$ with $w_1<w_2<n$ satisfies the gap vector condition for $v$.

Observe that gap vectors for a given prefix $p\in \S{k}$ form an upper order ideal in $\mathbb{N}^{k+1}$, since if $\vec{v}$ is a gap vector so is any $\vec{u}\geq \vec{v}$.  Hence it suffices to determine only the minimal elements (which form a basis).  For details on the discovery of gap vectors and automating the process, see \cite{Vatter2008, Zeilberger2006, Baxter2012}.

Note that if the prefix $p$ contains a pattern in $B$, then $\Spt{n}{B}{p}{w}=\emptyset$ for any appropriate $w$, and so $\vec{0}=\langle 0,0,\ldots, 0 \rangle$ is a gap vector.  


\subsection{Reversible Deletability}\label{rd:intro}

When $w$ fails the gap vector criterion for all gap vectors $\vec{v}$, we must rely on bijections with previously-computed $\Spt{\hat{n}}{B}{\hat{p}}{\hat{w}}$.  To continue our example above, consider $\Spt{n}{1\d2\d3}{12}{w_1 n}$.  Here $w_1 n$ fails all gap vector criteria, because $\langle 0,0,1\rangle$ forms the basis for the ideal of gap vectors and $\vec{g}(n,w_1 n)=\langle w_1-1, n-w_1-1,0 \rangle \not\geq \langle 0,0,1\rangle$.  However, any $\pi\in \Spt{n}{1\d2\d3}{12}{w_1 n}$  has $\pi_2=n$, so we may use the map $d_2:\ \pi_1 \pi_2 \dotsc \pi_n \mapsto \red(\pi_1 \pi_3 \dotsc \pi_n)$ to form a bijection
$\Spt{n}{1\d2\d3}{12}{w_1 n} \to \Spt{n-1}{1\d2\d3}{1}{w_1}$.  The deletion of a letter always preserves pattern-avoidance properties when considering classical patterns, but inverting the map by inserting a letter has the potential for creating a forbidden pattern.  Here, however, inserting an $n$ at the second index cannot possibly create a $1\d2\d3$, so we may safely reverse the deletion.

More generally define the deletion $d_r(\pi) := \red(\pi_1 \ldots \pi_{r-1} \pi_{r+1} \ldots \pi_n)$, that is, the permutation obtained by omitting the $r^{th}$ letter of $\pi$ and reducing.  Furthermore for a set $R$, define $d_R(\pi)$ to be the permutation obtained by deleting $\pi_r$ for each $r\in R$ and then reducing.  For a word $w$ with no repeated letters, define $d_r(w)$ be the word obtained by deleting the $r^{th}$ letter and then subtracting 1 from each remaining letter larger than $w_r$.  Similarly, to construct $d_R(w)$ delete $w_r$ for each $r\in R$ and subtract $\bigl|\{r\in R: w_r < w_i\}\bigr|$ from each remaining $w_i$.  For example $d_3(6348) =537$ and $d_{\{1,3\}}(6348) = 36$.  It can be seen that this definition is equivalent to the one given above when $w\in\S{k}$, and it allows for more succinct notation in the upcoming definition.  
In the unrestricted case, ${d_R: \Spt{n}{\emptyset}{p}{w} \to \Spt{n-|R|}{\emptyset}{d_R(p)}{d_R(w)}}$ is a bijection for any set $R\subseteq [|p|]$.  Sometimes we are lucky and the restriction to $\Spt{n}{B}{p}{w}$ is a bijection with $\Spt{n-|R|}{B}{d_R(p)}{d_R(w)}$, leading to the following definition:

\begin{definition}
The set of indices $R$ is \emph{reversibly deletable for $p$ with respect to $B$} if the map 
$$d_R: \Spt{n}{B}{p}{w} \to \Spt{n-|R|}{B}{d_R(p)}{d_R(w)} $$
is a bijection for all words $w$ failing the gap vector criterion for every gap vector of $p$ with respect to $B$ (\ie $d_R$ is a bijection for all $w$ such that $\Spt{n}{B}{p}{w}\neq \emptyset$).
\end{definition}

Note that the empty set $R=\emptyset$ is reversibly deletable for any $p$ and $B$, but is uninteresting.  Additionally, if $\vec{0}$ is a gap vector then any set $R\subseteq [ |p| ]$ is vacuously reversibly deletable since $\Spt{n}{B}{p}{w} = \emptyset$ for any prefix $w$.


Proving that a set is reversibly deletable for prefix $p$ with respect to $B$ can be carried out by a finite list of verifications, and thus can be done via computer.  This is proven in \cite{Vatter2008} for the case that $B$ contains only classical patterns, and in \cite{Baxter2012} in the case that $B$ contains vincular patterns.  The process of automated discovery itself is not relevant to the present work and will be omitted.

\subsection{Formal Definition of Enumeration Schemes}

We will now formally define an enumeration scheme.

\begin{definition}\label{def:scheme}
Let $B$ be a set of vincular patterns.  An \emph{enumeration scheme for $B$} is a set $E$ of triples $(p, G, R)$, where $p$ is a permutation (\ie, the prefix pattern), $G$ is a basis of gap vectors for $p$ with respect to $B$, and $R$ is a reversibly deletable set for $p$ with respect to $B$.  Furthermore, $E$ must satisfy the following criteria:

\begin{enumerate}
 \item $(\epsilon, \emptyset, \emptyset)\in E$.
 \item For each $(p, G, R) \in E$,
  \begin{enumerate}
   \item If $R=\emptyset$ and $\vec{0}\notin G$, then there exists a triple $(p', G', R') \in E$ for each child $p'$ of $p$.
   \item If $R\neq\emptyset$, then there exists a triple $(\hat{p}, \hat{G}, \hat{R})\in E$ for $\hat{p}=d_{R}(p)$.
  \end{enumerate}
\end{enumerate}
\end{definition}

To compute $\spt{n}{B}{p}{w}$ for a fixed $n$, $p$, and $w$, the enumeration scheme $E$ is ``read''  by finding the appropriate triple $(p, G, R) \in E$ and concluding:

\begin{enumerate}
  \item If $w$ satisfies the gap vector criteria for some $\vec{v}\in G$, then $\spt{n}{B}{p}{w}=0$.
  \item If $w$ fails the gap criteria for all $\vec{v}\in G$ and $R\neq \emptyset$, then $\spt{n}{B}{p}{w} = \spt{n-|R|}{B}{d_{R}(p)}{d_{R}(w)}$.
  \item If $w$ fails the gap criteria for all $\vec{v}\in G$ and $R = \emptyset$, then $\sp{n}{B}{p} = \sum\limits_{p'\in \Sp{k+1}{B}{p}} \sp{n}{B}{p'}$.
\end{enumerate}

When combined with the initial conditions that $\spt{n}{B}{p}{w}=1$ whenever $p$ has length $n$ and avoids $B$, the scheme provides a system of recurrences to compute $\spt{n}{B}{p}{w}$ and ultimately $\sav{n}{B}$.

To illustrate, consider the enumeration scheme for $\Sav{n}{1\d2\d3}$:
\begin{equation}\label{eqn:123scheme}
\left\{ (\epsilon, \emptyset, \emptyset), (1, \emptyset, \emptyset), (12, \{ \langle 0,0,1 \rangle \}, \{2\}), (21, \emptyset, \{1\}) \right\}
\end{equation}
Since $R_{\epsilon} = \emptyset$, the first condition above requires the presence of $(1, G_1, R_1)$.  Starting with the pattern 1 yields no additional information, so $R_1=\emptyset$ and thus explaining the presence of $(12,G_{12}, R_{12})$ and $(21,G_{21}, R_{21})$.  As discussed above, $\{\langle 0,0,1\rangle\}$ forms a basis for the gap vectors for 12, and whenever $w$ fails this gap vector criteria the second letter is reversibly deletable.  For the fourth entry in the scheme, suppose that $\pi\in \Sp{n}{\emptyset}{21}$ contains a $1\d2\d3$ pattern involving the first letter, say $\pi_1 < \pi_i < \pi_{j}$ for $i<j$.  Then since $\pi_2<\pi_1$, we see that $\pi_2 <\pi_i < \pi_j$ is another $1\d2\d3$ pattern.  Therefore $\pi_1$ cannot be the deciding factor for whether $\pi$ contains $1\d2\d3$.  Hence the index 1 is reversibly deletable, so $R_{21}=\{1\}$.  

Enumeration schemes exhibit a tree-like structure.  The empty prefix $\epsilon$ serves as the root, and the children of each prefix are drawn as children in a rooted tree.  When a prefix has nontrivial gap vector criteria, we list those basis vectors below it.  When prefix $p$ has a non-empty reversibly deletable set $R$, we draw an arrow from $p$ to $d_R(p)$ labeled with ``$d_R$''.  See Figure \ref{fig:123scheme} for an example.

\begin{figure}[htb]
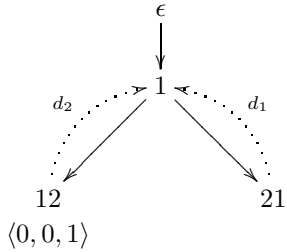

\centering
\xy 
(0,10)*+{\epsilon}="A";
(0,0)*+{1}="B";
(-15,-15)*+{12}="C1";
(-15,-20)*+{\langle 0,0,1 \rangle};
( 15,-15)*+{21}="C3";
{\ar@{->} "A"; "B"};
{\ar@{->} "B"; "C1"};
{\ar@{->} "B"; "C3"};
{\ar @/^1pc/ @{.>} ^{d_2} "C1"; "B"};
{\ar @/_1pc/ @{.>} _{d_1} "C3"; "B"};
\endxy
 \caption{Tree representation of the enumeration scheme for $\Sav{n}{1\d2\d3}$}
 \label{fig:123scheme}
\end{figure}
If $|E|$ is finite, we say that $B$ admits a finite enumeration scheme.  A finite enumeration scheme gives us a polynomial-time algorithm to compute $\spt{n}{B}{p}{w}$.  We construct the system of recurrences based on the partitions and bijections above, along with base cases as given by the gap vector criteria and the trivial cases when $\Sp{n}{B}{p}=\{p\}$ or $\emptyset$.  For example, the above enumeration scheme in \eqref{eqn:123scheme} translates into the following system of recurrences:

 \begin{equation}\label{eqn:123rec}
  \begin{split}
   \sav{n}{1\d2\d3}  &= \sp{n}{1\d2\d3}{\epsilon} \\
               &= \sp{n}{1\d2\d3}{1} \text{when $n>0$} \\
               &= \sum_{i=1}^n \spt{n}{1\d2\d3}{1}{i}  \\
   \spt{n}{1\d2\d3}{1}{i} &= \sum_{j=1}^{i-1} \spt{n}{1\d2\d3}{21}{ij} + \sum_{j=i+1}^{n} \spt{n}{1\d2\d3}{12}{ij}  \\
   \spt{n}{1\d2\d3}{12}{ij} &=
        \begin{cases}
                       0            & \text{if $n-j\geq 1$} \\
                 \spt{n}{1\d2\d3}{1}{i} & \text{otherwise} \\
                 \end{cases} \\
   \spt{n}{1\d2\d3}{21}{ij} &= \spt{n}{1\d2\d3}{1}{j}
  \end{split}
 \end{equation}

The recurrences in \eqref{eqn:123rec} simplify to create the following recurrence:

\begin{equation}\label{eqn:123recSimpl}
  \spt{n}{1\d2\d3}{1}{i} = \sum_{j=1}^{i} \spt{n-1}{1\d2\d3}{1}{j}
\end{equation}

Amont other things, one can then evaluate this recurrence by hand to identify the closed form $\sav{n}{1\d2\d3} = \frac{1}{n+1} \binom{2n}{n}$, the Catalan numbers.

The length of the longest prefix $p$ appearing in finite scheme $E$ is the called the \emph{depth} of $E$.  Not every set $B$ admits a finite enumeration scheme, the simplest example being the classical pattern $2\d3\d1$.  Let $E$ be the scheme for $\Sav{n}{2\d3\d1}$, and let $J_t = t(t-1)\dotsm 21$ be the decreasing permutation of length $t$.  It can be shown that for any $t$ there are no gap vectors for $J_t$ and no non-empty reversibly deletable sets.  Hence $E$ must contain the triple $(J_t, \emptyset, \emptyset)$ for each $t\geq 1$ and hence $E$ is infinite.  

It should be noted that the enumeration scheme for $\Sav{n}{1\d3\d2}$ is finite (of depth 2) and $\sav{n}{2\d3\d1}=\sav{n}{1\d3\d2}$ by symmetry.    More generally, it can be seen that if $B$ admits an enumeration scheme $E_B$ of depth $K$ then its set of complements $B^c = \{\sigma^c: \sigma \in B\}$ also admits an enumeration scheme $E_{B^c}$ of depth $K$.   The analogous statements regarding $B^r = \{\sigma^r: \sigma\in B\}$ do not hold and so $B$ may not have a finite scheme while $B^r$ does, as exhibited by $B=\{2\d3\d1\}$.

\section{Enumeration-scheme-compatible statistics}\label{sec:deletionfriend}

\subsection{Definitions and interaction with enumeration schemes}
The author proves in \cite{Baxter2010} that if $B$ admits a finite enumeration scheme, then the distribution of the statistic $\inv(\pi)$ over $\Sav{n}{B}$ can be computed via the same enumeration scheme.  This is the consequence of comparing the inversion number of a permutation and its image under the deletion map $d_R:\Spt{n}{B}{p}{w} \to \Spt{n-|R|}{B}{d_R(p)}{d_R(w)}$.   In particular, for $\pi\in\S{n}$ the change in the inversion number after deleting the $r^{th}$ letter is given by
\begin{equation}
  \delta^{\inv}_r(\pi) := \inv(\pi) - \inv(d_r(\pi)) = (\pi_r - 1) + \sum_{i<r} \sgn(\pi_i - \pi_r),
\end{equation}
where $\sgn(x)$ is the signum function:
\begin{equation*}
\sgn(x):=
\begin{cases}
  -1, & x<0 \\
   0, & x=0 \\
   1, & x>0.
\end{cases}
\end{equation*}
  For the more general case, let $R = \{r_1, \dotsc, r_t\}$ where $r_j < r_{j+1}$.  Then the deletion $d_R$ has the following effect on inversion number:
\begin{equation}
 \delta^{\inv}_{R}(\pi) := \inv(\pi) - \inv(d_R(\pi)) = \sum_{r \in R} \delta_r (\pi) - \inv(\pi_{r_1} \dotsm \pi_{r_t})
\end{equation}

Observe that $\delta_R(\pi)$ can be written purely in terms of the letters $\pi_1, \dotsc,  \pi_{r_t}$.  Therefore if $E$ is an enumeration scheme for pattern set $B$, and $(p, G, R)\in E$, then $\delta_R$ is constant over the set $\Spt{n}{B}{p}{w}$ for any fixed $w$.  Thus we define a new function $\Delta^{\inv}_R(w,n)$ on prefix words $w$, which takes on the value $\delta^{\inv}_R(\pi)$ given by a $\pi \in \Spt{n}{B}{p}{w}$.   Therefore we may recursively compute the distribution of $\inv$ over $\Spt{n}{B}{p}{w}$ via 
\begin{equation}
 F(\Spt{n}{B}{p}{w}, \inv,q) = q^{\Delta^{\inv}_R(w,n)} F(\Spt{n-|R|}{B}{d_R(p)}{d_R(w)}, \inv, q)
\end{equation}

The above results motivate the following definitions:

\begin{definition}\label{def:deltaR}
Let $f: \bigcup_{n\geq 0} \S{n} \to \mathbb{Z}$ be a permutation statistic.  For a permutation $\pi$  let the \emph{$R$-deletion difference}, denoted $\delta^{f}_{R}(\pi)$, be $  f(\pi) - f(d_R(\pi)) $.  

For nonnegative integer $m$, a permutation statistic $f$ is \emph{enumeration-scheme-compatible} (or ``ES-compatible'') \emph{with margin $m$} if for any positive integer $t$ and any $R\subseteq [t]$, the $R$-deletion differences $\delta^{f}_{R} (\pi) = \delta^{f}_{R} (\pi')$ whenever $\pi$ and $\pi'$ are two permutations of length $n\geq t+m$ such that $\pi_1 \dotsm \pi_{t+m} = \pi'_1 \dotsm \pi'_{t+m}$.  We denote this constant value $\Delta^{f}_{R}(w,n)$ where $w=\pi_1 \dotsm \pi_{t+m}$.
\end{definition}

In other words, permutation statistic $f$ is ES-compatible with margin $m$ if $\delta^{f}_{R}(\pi)$ may be determined from only the length of $\pi$ and its first $\max R+m$ letters.  Note that if $f$ is ES-compatible with margin $m$, then $f$ is also ES-compatible with margin $m'$ for any $m' \geq m$.  

The results from \cite{Baxter2010} cited above may be rephrased as follows:
\begin{theorem}\label{thm:InvIsDF}
 The inversion number is ES-compatible with margin $0$.
\end{theorem}

It follows from the definition of ES-compatible that enumeration schemes are amenable to computing the distribution for any ES-compatible statistic:

\begin{theorem}\label{thm:HowToRecurse}
  Let $f$ be a ES-compatible permutation statistic with margin $m$.   If $R$ is reversibly deletable for prefix $p\in \S{k}$ with respect to $B$ and $\max R + m \leq k$, then
\begin{equation} \label{eqn:HowToRecurse}
 F(\Spt{n}{B}{p}{w}, f, q) = q^{\Delta^{f}_{R}(w,n)} \, F(\Spt{n-|R|}{B}{d_R(p)}{d_R(w)}, f, q),
\end{equation}
where $\Delta^{f}_{R}(w,n)$ has the value $\delta^f_R(\pi)$ for any $\pi \in \Spt{n}{B}{p}{w}$.
\end{theorem}

\begin{proof}
Since $\max R + m \leq |p|$ and $f$ is ES-compatible with margin $m$, we see from the definition of ES-compatible that $\delta^f_R (\pi)$ is constant for all permutations $\pi \in \Spt{n}{B}{p}{w}$.  Therefore $\Delta^{f}_{R}(w,n)$ is well-defined.  Furthermore, $f(\pi) = \Delta^f_R(w,n) + f(d_R(\pi))$ for any $\pi \in \Spt{n}{B}{p}{w}$.

From the definition of reversibly-deletable, $d_R: \Spt{n}{B}{p}{w} \to \Spt{n-|R|}{B}{d_R(p)}{d_R(w)}$ is a bijection.  Shifting focus to the weight enumerators, we see that

\begin{equation}
 \begin{split}
    F(\Spt{n}{B}{p}{w}, f, q) :&= \sum_{\pi \in \Spt{n}{B}{p}{w}}  q^{f(\pi)} \\
                                              &= \sum_{\pi \in \Spt{n}{B}{p}{w}}  q^{\Delta^f_R(w,n)} q^{f(d_R(\pi))} \\
                                              &= q^{\Delta^f_R(w,n)} \sum_{\pi \in d_R(\Spt{n}{B}{p}{w})}   q^{f(\pi)} \\
                                              &= q^{\Delta^f_R(w,n)}  \sum_{\pi \in  \Spt{n-|R|}{B}{d_R(p)}{d_R(w)}}  q^{f(\pi)} \\
                                             &= q^{\Delta^f_R(w,n)} F(\Spt{n-|R|}{B}{d_R(p)}{d_R(w)}, f, q)
 \end{split}
\end{equation}
Thus we have proven equation \eqref{eqn:HowToRecurse}.
\end{proof}

By a similar proof we get the following multivariate generalization of Theorem \ref{thm:HowToRecurse}:

\begin{theorem}\label{thm:HowToRecurseMulti}
  Let $f_1, \dotsc, f_s$ be ES-compatible permutation statistics, each with margin at most $m$, and let $\mathbf{f}=\langle f_1, \dotsc, f_s \rangle$ and $\mathbf{q}=\langle q_1, \dotsc, q_s \rangle$.   If $R$ is reversibly deletable for prefix $p\in \S{k}$ with respect to $B$ and $\max R + m \leq k$, then
\begin{equation} \label{eqn:HowToRecurseMulti}
 F(\Spt{n}{B}{p}{w}, \mathbf{f}, \mathbf{q}) = \Bigl( q_1^{\Delta^{f_1}_{R}(w,n)} \dotsm q_s^{\Delta^{f_s}_{R}(w,n)} \Bigr) \, F(\Spt{n-|R|}{B}{d_R(p)}{d_R(w)}, \mathbf{f}, \mathbf{q}),
\end{equation}
where each $\Delta^{f_i}_{R}(w,n)$ has the value $\delta^{f_i}_R(\pi)$ for any $\pi \in \Spt{n}{B}{p}{w}$.
\end{theorem}

To more clearly tie Theorems \ref{thm:HowToRecurse} and \ref{thm:HowToRecurseMulti} to enumeration schemes, we introduce the following terminology:

\begin{definition}
 For nonnegative integer $c$, an enumeration scheme $E$ has \emph{clearance $c$} if for each $(p, G, R)\in E$, either $R = \emptyset$, $\vec{0} \in G$, or $|p| - \max R \geq c$.
\end{definition}

For example, the scheme for $1\d2\d3$-avoiding permutations given in \eqref{eqn:123scheme} has clearance $0$ because of the triple $(12, \{ \langle 0,0,1\rangle \}, \{2\})$.   The clearance of an enumeration scheme describes the largest margin that the scheme could accomodate, as detailed in the following corollaries.

\begin{corollary}\label{cor:SchemeComputing}
 If $f$ is a ES-compatible permutation statistic of margin $m$ and $E$ is an enumeration scheme for pattern set $B$ with clearance at least $m$, then $F(\Sav{n}{B}, f, q)$ may be computed in polynomial time (via enumeration scheme $E$).
\end{corollary}

\begin{corollary}\label{cor:SchemeComputingMulti}
   Let $f_1, \dotsc, f_s$ be ES-compatible permutation statistics, each with margin at most $m$, and let $\mathbf{f}=\langle f_1, \dotsc, f_s \rangle$ and $\mathbf{q}=\langle q_1, \dotsc, q_s \rangle$.  Let $E$ be a finite enumeration scheme for pattern set $B$  with clearance at least $m$.  Then $F(\Sav{n}{B}, \mathbf{f}, \mathbf{q})$ may be computed in polynomial time (via enumeration scheme $E$).
\end{corollary}

Clearly corollaries \ref{cor:SchemeComputing} and \ref{cor:SchemeComputingMulti} are impractical if there is no scheme $E$ satisfying the conditions stated.  It will be shown in Theorem \ref{thm:Deepening} of Section \ref{sec:Deepening} that a finite scheme of any clearance is sufficient for a polynomial time computation since a scheme can be ``deepened'' to create a scheme for the same pattern set with any desired clearance.  


\subsection{Examples of ES-compatible statistics}\label{sec:deletionfriendlyexamples}

We now take some time to prove some well-known permutaiton statistics are indeed ES-compatible.

\subsubsection{Copies of consecutive patterns}

We first consider statistics based on the number of copies of a given consecutive pattern.  Several well-studied statistics can be phrased in terms of the number of copies of certain consecutive patterns.  The descent number, $\des(\pi)$, is the number of copies of the consecutive pattern $21$.  The number of double-descents, \ie indices $i$ so that  $\pi_i > \pi_{i+1} > \pi_{i+2}$, is the number of copies of the consecutive pattern $321$.  In Subsection \ref{sec:peaks} we discuss the distribution for the number of peaks, \ie indices $i$ so that $\pi_{i-1} < \pi_{i}$ and $\pi_{i} > \pi_{i+1}$, which is the total of the number of copies of $132$ and the number of copies of $231$.

\begin{theorem}\label{thm:ConsAreDF}
 Let $\sigma\in\S{t}$ and $f(\pi)$ be the number of copies of the (consecutive) pattern $(\sigma, [t-1])$ in $\pi$.  Then $f$ is a ES-compatible statistic with margin $t-1$.
\end{theorem}

Note that $f(\pi)$ is defined for any word $\pi$, not just permutations.  Furthermore, if $\pi \oi \pi'$, then $f(\pi) = f(\pi')$.

\begin{proof}
 Fix pattern $\sigma\in\S{t}$ and let $k\geq 1$.  We will prove that $f(\pi) -  f(\pi_1\dotsm \pi_k) = f(d_R(\pi)) -   f(d_R(\pi_1\dotsm \pi_k))$ for any $R\subseteq [k-t+1]$.  From this it is clear that $\delta^f_R (\pi) = f(\pi_1 \dotsm \pi_k) - f(d_R(\pi_1 \dotsm \pi_k))$ for any $k\geq \max{R} + t-1$, and so $f$ is ES-compatible with margin $t-1$.

 Let $f_i(\pi)$ be the number of copies of $(\sigma, [t-1])$ starting at $\pi_i$.  Since $(\sigma, [t-1])$ is consecutive, we see that
\begin{equation*}
 f_i (\pi) = \chi[\pi_i \dotsm \pi_{i+t-1} \oi \sigma],
\end{equation*}
where $\chi[P]$ is the characteristic function for statement $P$, \ie, $\chi[P]$ equals $1$ if $P$ is true and $0$ otherwise.
Therefore $f(\pi) = \sum\limits_{i=1}^{n-t+1} f_i(\pi)$ for $\pi\in\S{n}$, and splitting this sum implies that:
\begin{equation*}
 \begin{split}
  f(\pi) &= \sum_{i=1}^{n-t+1} f_i(\pi) \\
           &= \sum_{i=1}^{k-t+1} f_i(\pi) + \sum_{i=k-t+2}^{n-t+1} f_i(\pi) \\
           &=  \sum_{i=1}^{k-t+1} f_i(\pi_1 \dotsm \pi_{k}) + \sum_{i=k-t+2}^{n-t+1} f_i(\pi_{k-t+2} \dotsm \pi_n) \\
           &= f(\pi_1 \dotsm \pi_k) + f(\pi_{k-t+2} \dotsm \pi_n)
 \end{split}
\end{equation*}
  Thus $f(\pi) =f(\pi_1 \dotsm \pi_k) + f(\pi_{k-t+2} \dotsm \pi_n) $.  Similarly, let $\pi' = d_R(\pi)$ for $R\subseteq [k-t+1]$, and by the appropriate sum-splitting we see that $f(\pi') =f(\pi'_1 \dotsm \pi'_{k-|R|}) + f(\pi'_{k-t+2-|R|} \dotsm \pi'_{n-|R|}) $.  By the definition of $d_R$, $ \pi_{k-t+2} \dotsm \pi_n \oi \pi'_{k-t+2-|R|} \dotsm \pi'_{n-|R|}$.  Hence it follows that
\begin{equation*}
 \begin{split}
   f(\pi) -  f(\pi_1\dotsm \pi_k) &= f(\pi_{k-t+2} \dotsm \pi_n) \\
                                                &= f(\pi'_{k-t+2-|R|} \dotsm \pi'_{n-|R|})\\
                                                &= f(\pi') - f(\pi'_1 \dotsm \pi'_{k-|R|})
 \end{split}
\end{equation*}

Thus we have confirmed $f(\pi) -  f(\pi_1\dotsm \pi_k) = f(d_R(\pi)) -   f(d_R(\pi_1\dotsm \pi_k))$ and our result follows.
\end{proof}

\begin{remark}
In terms relevant to enumeration schemes, for any permutation $\pi\in \Spt{n}{B}{p}{w}$ for a prefix of length $k$, 
\begin{equation}
 \Delta^f_R(w,n) = f(w) - f(d_R(w))= f(p) - f(d_R(p)),
\end{equation}
where $f$ counts the number of copies of a consecutive pattern.
\end{remark}

\subsubsection{Copies of vincular patterns}

We next consider the number of copies of a vincular pattern of the form  $\sigma_1 \dotsm \sigma_{t-1} \d \sigma_t$.  The proof will proceed similarly to that of Theorem \ref{thm:ConsAreDF}.  Note that the inversion number $\inv(\pi)$ is the number of copies of $2\d1$, and so is a special case of this result.  Subsection \ref{sec:majind} extends the results in this section to apply the major index statistic.

\begin{theorem}\label{thm:tail}
  Let $\sigma\in\S{t}$ and let $g(\pi)$ be the number of copies of the pattern $(\sigma, [t-2])$ in $\pi$.  Then $g$ is a ES-compatible statistic with margin $t-2$.
\end{theorem}

\begin{proof}
   Fix $\sigma\in\S{t}$ and let $k\geq 1$.  We will prove that $\delta^g_R(\pi) := g(\pi) - g(d_R(\pi))$ is determined by the length of $\pi$ together with the prefix $\pi_1 \dotsm \pi_k$ for any $R\subseteq [k-t+2]$.  From this it follows that $g$ is ES-compatible with margin $t-2$.  

For a permutation $\pi\in\S{n}$, define $g_i(\pi)$ to be the number of copies of $(\sigma, [t-2])$ starting at $\pi_i$. Clearly $g(\pi) = \sum_{i=1}^{n} g_i(\pi)$, and if $\pi$ and $\pi'$ are order-isomorphic words then $g_i(\pi) = g_i(\pi')$.   In particular we see that $g_i(\pi)$ can be computed in the following way based, on which $\pi_j$ can be the last letter of a copy of $(\sigma, [t-2])$ starting at $\pi_i$:
\begin{equation}
 \begin{split}
 g_i(\pi) :=& \bigl| \{ j: j>i+t-2, \text{ and } \pi_i \dotsm \pi_{i+t-2}\;\pi_j \oi \sigma  \} \bigr|\\
              =& \bigl| \{ j: j\geq 1 \text{ and }  \pi_i \dotsm \pi_{i+t-2}\;\pi_j \oi \sigma  \}\bigr| - \bigl| \{j: j<i \text{ and } \pi_i \dotsm \pi_{i+t-2}\;\pi_j \oi \sigma \} \bigr|
 \end{split}
\label{eqn:gi}
\end{equation}

We now will show that $g_i(\pi)$ can be determined entirely from $|\pi|$ and $\pi_1 \dotsm \pi_{i+t-2}$ by showing each of the addends in \eqref{eqn:gi} requires such limited information.

We consider the first term, $ \bigl| \{ j: j\geq 1 \text{ and }  \pi_i \dotsm \pi_{i+t-2}\;\pi_j \oi \sigma  \}\bigr|$.  Given the fixed $\sigma\in\S{t}$, we can define $h_i(w, n)$ for words $w\in[n]^k$ as follows:
\begin{equation}
 h_i(w, n) := 
 \begin{cases}
   \min(w_i, \dotsc, w_{i+t-2}) -1 & \text{if } \sigma_t = 1 \\
    n- \max(w_i, \dotsc, w_{i+t-2}) & \text{if } \sigma_t = t \\
  w_{b} -   w_{a} -1 & \text{if } 1 < \sigma_t < t, \sigma_a = \sigma_{t} -1, \sigma_b = \sigma_{t}+1.
 \end{cases}
\end{equation}
  If $\sigma_t = 1$, then $h_i(\pi,n)$ yields the number of letters $\pi_j$ which are less than each of $\pi_i, \dotsc , \pi_{i+t-2}$, in which case  $\pi_i \dotsm \pi_{i+t-2}\;\pi_j \oi \sigma$.  Similarly $\sigma_t = t$, then $h_i(\pi,n)$ yields the number of letters $\pi_j$ which are greater than each of $\pi_i, \dotsc , \pi_{i+t-2}$.  Last, if $a$ and $b$ are defined by $\sigma_{a} = \sigma_{t}-1$ and $\sigma_{b} = \sigma_{t}+1$, then $h_i(\pi,n)$ yields the number of letters $\pi_j$ so that $\pi_{i+a-1} < \pi_j < \pi_{i+b-1}$ since every number between the values $\pi_{i+a-1}$ and $\pi_{i+b-1}$ appears somewhere in $\pi$.  In each case, we see that $$h_i(\pi, n) = \bigl| \{ j: j \geq 1 \text{ and } \pi_i \dotsm \pi_{i+t-2}\;\pi_j \oi \sigma  \}\bigr|.$$  Observe that $h_i(\pi, n) = h_i(\pi_1 \dotsm \pi_{s}, n)$ for any $s \geq i+t-2$, and so the first term of the sum in \eqref{eqn:gi} is determined solely by $|\pi|$ and $\pi_i \dotsm \pi_{i+t-2}$.

Similarly, the term $\bigl| \{j: j<i \text{ and } \pi_i \dotsm \pi_{i+t-2}\;\pi_j \oi \sigma \} \bigr|$ can be determined entirely by $\pi_1 \dotsm \pi_{i+t-2}$.   Define $h'_i(w)$ for word $w$ to be
\begin{equation*}
 h'_i(w) := \sum_{j=1}^{i-1} \chi[w_i \dotsm w_{i+t-2}\;w_j \oi \sigma],
\end{equation*}
and so it is clear that $h'_i(\pi) = \bigl| \{j: j<i \text{ and } \pi_i \dotsm \pi_{i+t-2}\;\pi_j \oi \sigma \} \bigr|$.  Note that $h'_i(\pi) = h'_i(\pi_1 \dotsm \pi_{s})$ for any $s \geq i+t-2$.

Combining the above  observations, we can decompose the sum   $g(\pi) = \sum_{i=1}^{n} g_i(\pi)$ as follows:
\begin{equation}\label{eqn:gisum}
 \begin{split}
  g(\pi) &= \sum_{i=1}^{n} g_i(\pi) \\
          &=  \sum_{i=1}^{k-t+2} g_i(\pi) + \sum_{i=k-t+3}^{n} g_i(\pi) \\
          &=  \sum_{i=1}^{k-t+2} \bigl( h_i(\pi_1 \dotsm \pi_k) - h'_i(\pi_1 \dotsm \pi_k) \bigr) + \sum_{i=k-t+3}^{n} g_i(\pi) \\
 \end{split}
\end{equation}
Fix $R\subseteq [k-t+2]$ and let $\pi' = d_R(\pi)$.  The above observations also apply to $\pi'$ to imply:
\begin{equation}\label{eqn:g'isum}
  g(\pi') = \sum_{i=1}^{k-t-|R|+2} \bigl( h_i(\pi'_1 \dotsm \pi'_{k-|R|}, n-|R|) - h'_i(\pi'_1 \dotsm \pi'_{k-|R|}) \bigr) + \sum_{i=k-t-|R|+3}^{n-|R|} g_i(\pi') 
\end{equation}

We now focus on the last terms of the equations \eqref{eqn:gisum} and \eqref{eqn:g'isum}.  By the definition of $d_R$,  $\pi_{k-t+3} \dotsm \pi_n \oi \pi'_{k-t+3-|R|} \dotsm \pi'_{n-|R|}$.  Therefore $g_i(\pi) = g_{i-|R|}(\pi')$ for $i\geq k-t+3$, and so $\sum\limits_{i=k-t+3}^{n} g_i(\pi) = \sum\limits_{i=k-t-|R|+3}^{n-|R|} g_i(\pi') $.  Thus subtracting equation \eqref{eqn:g'isum} from \eqref{eqn:gisum} we see that:
\begin{equation}
 \begin{split}
 \delta^g_R(\pi) &=  g(\pi) - g(\pi') \\
                           &=  \sum_{i=1}^{k-t+2} \bigl( h_i(\pi_1 \dotsm \pi_k, n) - h'_i(\pi_1 \dotsm \pi_k) \bigr) -  \sum_{i=1}^{k-t-|R|+2} \bigl( h_i(\pi'_1 \dotsm \pi'_{k-|R|}, n-|R|) - h'_i(\pi'_1 \dotsm \pi'_{k-|R|}) \bigr) \\
 \end{split}
\end{equation}
Thus we have proven that $g$ is ES-compatible with margin $t-2$.
\end{proof}

\begin{remark}
In terms relevant to enumeration schemes, for any permutation $\pi\in \Spt{n}{B}{p}{w}$ for a prefix of length $k$, 
\begin{equation}
 \Delta^g_R(w,n) =\sum_{i=1}^{k-t+2} \bigl( h_i(w,n) - h'_i(w) \bigr)  - \sum_{i=1}^{k-t-|R|+2} \bigl( h_i(d_R(w), n-|R|) - h'_i(d_R(w)) \bigr) .
\end{equation}
\end{remark}

As an example of $h_i$ and $h'_i$ in practice, consider $\sigma=4123$.  Here $g(\pi)$ counts the number of copies of the vincular pattern $412\d3$.  Then $a=1$ and $b=3$. Let $\pi$ be any one of the $4!$ permutations of length $9$ such that $w = \pi_1 \dotsm \pi_5 = 86913$.  Since $\pi_3\pi_4\pi_5 = 913 \oi 413 = \sigma_1 \sigma_2 \sigma_3$, we see that $h_3(\pi, 9) = \pi_3 - \pi_5 -1 = 5$.  Also $h'_3(\pi) = 2$, since both 6 and 8 appear before $\pi_3 \pi_4 \pi_5$ and have values which lie between $\pi_5=3$ and $\pi_3=9$.  Therefore $g_3(\pi) = h_3(\pi) - h'_3(\pi) = 3$, and indeed there are 3 copies of $412\d3$ starting at $\pi_3$ (specifically, these copies are 9134, 9135, and 9137, since $\{4, 5, 7\} \subseteq \{\pi_6, \dotsc, \pi_9\}$).  In this example $g_1(\pi) = g_2(\pi) = 0$.  If $R=\{2,3\}$, which respects the needed margin of $2$ for a length-5 prefix, we get that $\pi' = d_R(\pi)$ has prefix $d_R(86913) = 713$.  Deleting the $6$ and $9$ from $\pi$ deletes the three copies of $412\d3$ described above, but creates three new copies since $g_1(\pi') = h_1(713,7) - h'_1(713) = 3 - 0 = 3$ (specifically, the new copies are witnessed by 7134, 7135, and 7136).  Hence in this case $\delta_R^g(\pi) = 0$.

Comparing Theorems \ref{thm:ConsAreDF} and \ref{thm:tail}, one might guess the trend continues, \ie that patterns of the form $(\sigma, [t-3])$ for $\sigma\in\S{t}$ are ES-compatible with margin $t-3$.  An extension involving partially-ordered generalized patterns, as introduced by Kitaev in \cite{Kitaev2005}, provides the proper generalization.  For example, a copy of the pattern $124\d3'\d3''$ would be witnessed by a copy of either $125\d3\d4$ or $125\d4\d3$.  Such a statistic of the form $$\sigma_1~\dotsc~\sigma_{t-1}~\d~\sigma_{t}'~\d~\sigma_{t}''~\d~\dotsc~\d~\sigma_{t}'''$$ can be seen to be ES-compatible with margin $t-2$ (one less than the length of the consecutive portion).  

\subsubsection{Right-to-left statistics}

A letter $w_i$ of word $w$ is a right-to-left maximum [resp., minimum] if $w_i > w_j$ [resp., $w_i < w_j$] for all $j>i$.  Let $\rtlmax(w)$ be the number of \rtl maxima in $w$ and $\rtlmin(w)$ be the number of \rtl minima in word $w$.    For example, if $\pi=28674153$, we see $\rtlmax(\pi) = 4$ (for $\pi_2$, $\pi_4$, $\pi_7$ and $\pi_8$) and $\rtlmin(\pi) = 2$ (for $\pi_6$ and $\pi_8$).  

In this subsection we prove the following theorem:
\begin{theorem}
The statistics $\rtlmin$ and $\rtlmax$ are ES-compatible with margin $0$.
\end{theorem}

It will be useful to have the following characterization of the \rtl minima and maxima for a permutation, which are based solely on prefixes.  The proof follows directly from the definition above and is omitted.
\begin{lemma}\label{lem:rtlchar}
  Let $\pi=\pi_1\dotsm\pi_n$ be a permutation.  Then,
\begin{enumerate}
\item $\pi_i$ is a \rtl minimum of $\pi$ if and only if $\{1,2,\dotsc, \pi_{i}-1\} \subseteq \{\pi_1, \pi_2, \dotsc, \pi_{i-1}\}$ (\textit{\ie}, all numbers less than $\pi_i$ lie to the left of $\pi_i$).

\item $\pi_i$ is a \rtl maximum of $\pi$ if and only if $\{\pi_{i}+1, \pi_{i}+2, \dotsc, n\} \subseteq \{\pi_1, \pi_2, \dotsc, \pi_{i-1}\}$ (\textit{\ie}, all numbers greater than $\pi_i$ lie to the left of $\pi_i$).

\end{enumerate}
\end{lemma}

For the remainder of the section, we will restrict ourselves to the proof that $\rtlmax$ is ES-compatible.  The proof that $\rtlmin$ is ES-compatible is analogous.

For integers $a$ and $b$ let $\rtlmax_{[a,b]}(\pi)$ be the number of \rtl maxima $\pi_i$ of $\pi$ such that $i$ is in the closed interval $[a,b]$.  For example, $\rtlmax_{[1,5]}(2867153)= 2$ (for $\pi_2$ and $\pi_5$).  Note that while $\pi_5$ is a \rtl maximum of $\pi_1 \dotsm \pi_5=28671$, it is not counted since this function only counts those letters which are \rtl maxima in the overall permutation.  We may decompose $\rtlmax(\pi)$ for any $\pi\in\S{n}$ and $1\leq t\leq n$, by $\rtlmax(\pi) = \rtlmax_{[1,t]}(\pi) + \rtlmax_{[t+1,n]}(\pi)$.  Thus it follows from the first half of Definition \ref{def:deltaR} that if $\pi\in\S{n}$, $R\subseteq [t]$ and $\pi':=d_R(\pi)\in\S{n-|R|}$,
\begin{equation}
\begin{split}
  \delta_{R}^{\rtlmax}(\pi) &= \rtlmax(\pi)  - \rtlmax(\pi') \\
                                            &= \bigl( \rtlmax_{[1,t]}(\pi) + \rtlmax_{[t+1,n]}(\pi) \bigr) - \bigl( \rtlmax_{[1,t-|R|]}(\pi') + \rtlmax_{[t-|R|+1,n-|R|]}(\pi') \bigr)
\end{split}
\end{equation}
Rearranging terms leaves us with
\begin{equation*}
  \delta_{R}^{\rtlmax}(\pi)    = \bigl( \rtlmax_{[1,t]}(\pi) -  \rtlmax_{[1,t-|R|]}(\pi') \bigr) + \bigl( \rtlmax_{[t+1,n]}(\pi) -  \rtlmax_{[t-|R|+1,n-|R|]}(\pi') \bigr) 
\end{equation*}

By the original definition of $\rtlmax$, it is clear the $\rtlmax_{[a,n]}(\pi) = \rtlmax(\pi_{a} \pi_{a+1} \dotsm \pi_n)$ for any $a$.  Since $\pi'_{t-|R|+1} \pi'_{t-|R|+2} \dotsm \pi'_{n-|R|} \oi \pi_{t+1}\pi_{t+2} \dotsm \pi_n$, it follows that $\rtlmax_{[t+1,n]}(\pi) -  \rtlmax_{[t-|R|+1,n-|R|]}(\pi') = 0$.  Thus we see that
\begin{equation}\label{eqn:deltartlmax}
  \delta_{R}^{\rtlmax}(\pi)    = \rtlmax_{[1,t]}(\pi) -  \rtlmax_{[1,t-|R|]}(\pi')
\end{equation}

Let $w$ be a word in $[n]^k$ without repeated letters, and define 
\begin{equation}
 \rtlmax^{*}(w,n) := \bigl| \bigl\{w_i:  \{w_{i} +1, w_{i}+2, \dotsc, n\} \subseteq \{w_{1}, w_{2}, \dotsc, w_{i-1}\} \bigr\}  \bigr|.
\end{equation}
  By Lemma \ref{lem:rtlchar} above immediately see for any $\pi\in\S{n}$ 
\begin{equation*}
   \rtlmax_{[1,t]}(\pi) = \rtlmax^{*}(\pi_1 \pi_2 \dotsm \pi_t, n).
\end{equation*}
Hence equation \eqref{eqn:deltartlmax} becomes
\begin{equation}
  \delta_{R}^{\rtlmax}(\pi)    = \rtlmax^{*}(\pi_1 \pi_2 \dotsm \pi_t, n) -  \rtlmax^{*}(\pi'_1 \pi'_2 \dotsm \pi'_{t-|R|}, n-|R|).
\end{equation}

Therefore if $R\subseteq [t]$, then $\delta_{R}^{\rtlmax}(\pi)$ depends only the values of $\pi_1 \dotsm \pi_t$.  Thus by Definition \ref{def:deltaR} we see that $\rtlmax$ is ES-compatible with margin $0$, and the proof is complete.   As mentioned previously, the proof that $\rtlmin$ is ES-compatible with $0$ proceeds analogously, where $\rtlmin^{*}(w,n):=~ \bigl| \bigl\{w_i:~ \{1, 2, \dotsc, w_{i}-1\} \subseteq \{w_{1}, w_{2}, \dotsc, w_{i-1}\} \bigr\}  \bigr|$.

In \cite{Baxter2013} the author generalizes right-to-left maxima with the \emph{right-to-left maximal copy} of a consecutive pattern $\sigma$.  For a permutation $\pi$ and consecutive pattern $\sigma$, the subfactor $\pi_i \pi_{i+1} \dotsm \pi_{i+k-1}$ is a \emph{right-to-left maximal copy of $\sigma$} if the following criteria are satisfied:
 \begin{enumerate}
  \item  $\pi_i \pi_{i+1} \dotsm \pi_{i+k-1} \oi \sigma$, and 
  \item if $j>i$ and $\pi_j \pi_{j+1} \dotsm \pi_{j+k-1} \oi \sigma$ and $\sigma_m = \min(\sigma_1, \dotsc, \sigma_k)$, then $\pi_{j+m-1} < \pi_{i+m-1} $.  In other words, the minimal letter of $\pi_i \pi_{i+1} \dotsm \pi_{i+k-1}$ is greater than the minimal letter of any other copy of $\sigma$ to starting the right of $\pi_i$.
 \end{enumerate}
For example, the permutation $31856742$ has four copies of the consecutive pattern $21$ (namely, $31$, $85$, $74$, and $42$), but only three of these copies (all but $31$) are right-to-left maximal.  The classical right-to-left maxima can be viewed as right-to-left maximal copies of the pattern $1$.  A straightforward generalization of the above argument proves that the statistic counting the number of right-to-left maximal copies of a consecutive pattern $\sigma$ of length $t$ is ES-compatible with margin $t-1$.

Before closing this section, it should be noted that the number of left-to-right maxima and left-to-right minima are \emph{not} ES-compatible.  For example, if $n=4$, $w=12$, and $R=\{1,2\}$ then $\delta_{R}^{\ltrmin}(1234) = 0$ while $\delta_{R}^{\ltrmin}(1243)=-1$.  For $n=4$, $w=43$, and $R=\{1,2\}$ we also see $\delta_{R}^{\ltrmax}(4321) = 0$ while $\delta_{R}^{\ltrmax}(4312) = -1$.

\section{Deepening Schemes}\label{sec:Deepening} 


Suppose that pattern set $B$ admits a finite enumeration scheme.  The question remains whether one can find a finite enumeration scheme for $B$ with clearance sufficient to accomodate a given ES-compatible statistic with margin $m$.  The algorithms from \cite{Baxter2012} can be altered to ensure that any constructed scheme has clearance $c$ \emph{if} such a scheme exists. The existence of such a scheme is guaranteed in the following theorem.

\begin{theorem}\label{thm:Deepening}
 Suppose the pattern set $B$ has a finite enumeration scheme $E$.  Then for any $c \geq 0$ there is a finite enumeration scheme $E'$ with clearance $c$.
\end{theorem}

To prove Theorem \ref{thm:Deepening}, we will first prove a lemma regarding reversibly deletable sets.

\begin{lemma}\label{lem:rdDeepening}
 Suppose that $R$ is a reversibly deletable set for prefix $p\in\S{k}$ with respect to $B$.  Then $R$ is also reversibly deletable for any permutation $p' \in \S{k'}$ such that $p'_1 \dotsm p'_k \oi p$, where $k'>k$.
\end{lemma}

\begin{proof}
If $R\subseteq[k]$ is reversibly deletable for $p\in \S{k}$, then $d_R: \Spt{n}{B}{p}{w} \to \Spt{n-|R|}{B}{d_R(p)}{d_R(w)}$ is a bijection whenever $\Spt{n}{B}{p}{w} \neq \emptyset$.  Let $k'>k$ and let $p'\in\S{k'}$ such that $\red(p'_1 \dotsm p'_k) = p$.  Let $w'\in [n]^{k'}$ so that $\Spt{n}{B}{p'}{w'}\neq \emptyset$.  We must show that $d_R: \Spt{n}{B}{p'}{w'} \to \Spt{n-|R|}{B}{d_R(p')}{d_R(w')}$ is a bijection.

First note that $\Spt{n}{B}{p'}{w'} = \Spt{n}{B}{p}{w} \cap \Spt{n}{\emptyset}{p'}{w'}$, where $w=w'_1 \dotsm w'_k$.  Since $\Spt{n}{B}{p'}{w'}\neq \emptyset$ and $\Spt{n}{B}{p'}{w'} \subseteq \Spt{n}{B}{p}{w}$, we know that $\Spt{n}{B}{p}{w} \neq \emptyset$.  Since $R$ is reversibly deletable for $p$, the map $d_R: \Spt{n}{B}{p}{w} \to \Spt{n-|R|}{B}{d_R(p)}{d_R(w)}$ is bijective.  For the unrestricted case, $d_R$ is a bijection from $\Spt{n}{\emptyset}{p'}{w'}$ to $ \Spt{n-|R|}{\emptyset}{d_R(p')}{d_R(w')}$.  Therefore $d_R$ is a bijection from $\Spt{n}{B}{p'}{w'}$ to $ \Spt{n-|R|}{B}{d_R(p)}{d_R(w)} \cap \Spt{n-|R|}{\emptyset}{d_R(p')}{d_R(w')}$.  Now it remains to show $\Spt{n-|R|}{B}{d_R(p')}{d_R(w')} =   \Spt{n-|R|}{B}{d_R(p)}{d_R(w)} \cap \Spt{n-|R|}{\emptyset}{d_R(p')}{d_R(w')}$ to complete our proof.  By the action of $d_R$, the word formed by the first $k-|R|$ letters of $d_R(w')$ is exactly $d_R(w)$, and so $\Spt{n-|R|}{B}{d_R(p')}{d_R(w')} \subseteq \Spt{n-|R|}{B}{d_R(p)}{d_R(w)}$, and the remaining inclusions are obvious from the definitions.  Thus $R$ is reversibly deletable for $p'$.
\end{proof}

Note that the resulting set $R$ in Lemma \ref{lem:rdDeepening} is not necessarily a maximal reversibly-deletable set for $p$.  For example, recall from Equation \eqref{eqn:123scheme} that $\{2\}$ is reversibly deletable for the prefix $12$ with respect to $B=\{1\d2\d3\}$.  Lemma \ref{lem:rdDeepening} implies that $\{2\}$ is reversibly deletable for the prefix $231$, although the larger set $\{1, 2\}$ is also reversibly deletable for $231$.


We are now ready to prove Theorem \ref{thm:Deepening}.
\begin{proof}[Proof of Theorem \ref{thm:Deepening}]

Let $E$ be a finite enumeration scheme for $B$, with depth $K$.  We will construct a scheme $E'$ with clearance $c \geq 1$.  If $c=0$, then $E$ will suffice since any enumeration scheme has clearance $0$.

We will construct a (finite) set $E'$ by creating a triple $(p, G(p), R(p)) \in E'$ for each $p \in \bigcup_{k = 0}^{K+c} \S{k}$.  For $p=\epsilon$, let $G(\epsilon) = \emptyset$ and $R(\epsilon)=\emptyset$, so we see $E'$ satisfies criterion 1 in Definition \ref{def:scheme}.  For $p \neq \epsilon$, let $G(p)$ be a basis of gap vectors for $p$ with respect to $B$, which may be constructed according to the algorithm described in \cite{Baxter2012}.  

We now construct $R(p)$.  If $|p| < K+c$, then let $R(p)= \emptyset$.  If $|p|= K+c$, then let $p'$ be the longest prefix $p'=\red(p_1 \dotsm p_s)$ such that there is a triple $(p', G', R')$ in the original scheme $E$.   Then $|p'|\leq K$, and since no child of $p'$ has a triple in $E$ (by maximality of $p'$) we know $R'$ is nonempty. By Lemma \ref{lem:rdDeepening}, $R'$ is also a reversibly deletable set for $p$, so we let $R(p) = R'$.  Furthermore, $R(p) = R' \subseteq [K]$, so $|p| - \max R(p) \geq c$ and so we see that $E'$ has clearance $c$.

We now verify that $E'$ satisfies the criteria to be an enumeration scheme for $B$, as outlined in Definition \ref{def:scheme}.  Each triple $(p, G(p), R(p))\in E'$ is constructed so that $G(p)$ is a basis of gap vectors for $p$ with respect to $B$ and so that $R(p)$ is a reversibly deletable set for $p$ with respect to $B$.  As previously mentioned, $E'$ satisfies criterion 1 since $(\epsilon, \emptyset, \emptyset) \in E'$.  If $R(p) = \emptyset$, then $|p|<K+c$ and so $E'$ contains a triple $(p', G(p'), R(p'))$ for each child $p'$ since $E'$ contains a triple for \emph{every} permutation of length $|p|+1$.  Therefore $E'$ satisfies criterion 2a.  If $R(p)\neq\emptyset$, then $|p|=K+c$ and $E'$ contains a triple $(\hat{p}, G(\hat{p}), R(\hat{p}))$ for $\hat{p} = d_{R(p)}(p)$ since $E'$ contains a triple for \emph{every} permutation with length less than $K+c$.  Therefore $E'$ satisfies criterion 2b.
\end{proof}

It perhaps goes without saying that $E'$ from the proof of Theorem \ref{thm:Deepening} is not usually minimal, neither in terms of number of triples nor the encoded recurrence.  For example, the proof constructs the following scheme $E'$ with clearance $1$ for $B=\{1\d2\d3\}$ based on the scheme $E$ in \eqref{eqn:123scheme}:
\begin{equation}
 \begin{split}
E' = \bigl\{ &
(\epsilon, \emptyset, \emptyset), 
(1, \emptyset, \emptyset), 
(12, \{\langle 0,0,1\rangle \}, \emptyset), 
(21, \emptyset, \emptyset), \\
&
(123, \{ \langle 0,0,0,0 \rangle \}, \{ 2\}),
(132, \{ \langle 0,0,1,0 \rangle, \langle 0,0,0,1 \rangle \}, \{2\} ), 
(231, \{ \langle 0,0,0,1 \rangle \}, \{ 2\}), \\
&
(213, \{\langle 0,0,0,1 \rangle \}, \{ 1\}), 
(312, \{ \langle 0,0,1,0 \rangle, \langle 0,0,0,1 \rangle \}, \{1\} ), 
(321, \emptyset, \{1\})
\bigr\}.
 \end{split}
\end{equation}
In practice, therefore, it is better to alter the automated-discovery algorithms from \cite{Baxter2012} to construct \emph{ab initio} a reversibly deletable set for each prefix $p$ so that only sets $R \subseteq [|p| - c]$ are considered.  This change guarantees the clearance criterion holds, and Theorem \ref{thm:Deepening} guarantees that the algorithm will succeed in finding a finite scheme whenever the algorithms would succeed without the clearance conditions.  Such an approach yields the following scheme with clearance $1$ for $B=\{1\d2\d3\}$.
\begin{equation}
 \begin{split}
E' = \bigl\{ &
(\epsilon, \emptyset, \emptyset), 
(1, \emptyset, \emptyset), 
(12, \{\langle 0,0,1\rangle \}, \emptyset), 
(21, \emptyset, \{1\}), \\ 
& (123, \{ \langle 0,0,0,0 \rangle \}, \{1, 2\}),
(132, \{ \langle 0,0,1,0 \rangle, \langle 0,0,0,1 \rangle \}, \{2\} ), 
(231, \{ \langle 0,0,0,1 \rangle \}, \{1, 2\})
\bigr\}.
 \end{split}
\label{eqn:123c1}
\end{equation}
Note in particular that the children of $21$ do not need to appear in a scheme with clearance $1$, since $21$ has a nonempty reversibly deletable set which respects the clearance requirement.

We close this section with a corollary combining Theorem \ref{thm:Deepening} and Corollary \ref{cor:SchemeComputing}. 

\begin{corollary} \label{cor:SchemeComputing2}
  If $f$ is a ES-compatible permutation statistic of margin $m$ and $E$ is an enumeration scheme for pattern set $B$, then $F(\Sav{n}{B}, f, q)$ may be computed in polynomial time.
\end{corollary}

\section{Applications}\label{sec:Applications}


We wish to take some time in this section to highlight some of the less obvious applications of the results above.  Studying statistics over sets $\Sav{n}{B}$ is relatively new, so much of what follows only scratches the surface.  In what follows we use the common notation that $(\sigma)(\pi)$ denotes the number of copies of vincular pattern $\sigma$ in permutation $\pi$.

\subsection{Implementation}\label{sec:implementation}

The above algorithms have been implemented in a Maple package \texttt{Statter}.  This package supercedes the package \texttt{gVatter} accompanying \cite{Baxter2012} and can perform the following tasks.
 \begin{enumerate}
   \item Build an enumeration scheme for a given pattern set $B$ with given clearance $c$.  (This also requires search parameters for the maximum size of gap vectors and maximum depth of prefixes.)
   \item Read a given enumeration scheme to get the distribution over $\Sav{n}{B}$ for given $n$ of a statistic based on the number of copies of a consecutive pattern, a pattern of type $\sigma_{1}\dotsc\sigma_{t-1}\d\sigma_{t}$, right-to-left maxima, or right-to-left minima. 
  \item Read a given enumeration scheme to get the distribution of multistatistics given above.
\end{enumerate}

For example, \texttt{Statter} gave the the enumeration scheme for pattern set $B=\{1\d2\d3\}$ with clearance 1 shown above in \eqref{eqn:123c1}.  From that, \texttt{Statter} can quickly give the distribution of the descent statistic over, say, $\Sav{10}{1\d2\d3}$:
$$
F(\Sav{10}{1\d2\d3},\des,q) = 42\,{q}^{4}+1770\,{q}^{5}+7515\,{q}^{6}+6455\,{q}^{7}+1013\,{q}^{8}+{q
}^{9}.
$$

A fuller investigation of the descent statistic over $\Sav{n}{1\d2\d3}$, particularly its connection to Dyck paths, is given in \cite{Barnabei2010}.

The package \texttt{Statter} is available for download from the author's homepage.

\subsection{Peaks and Valleys}\label{sec:peaks}

A \emph{peak} of a permutation $\pi$ is a letter $\pi_i$ such that $\pi_{i-1} < \pi_i > \pi_{i+1}$.  Let $\peak(\pi)$ be the number of peaks of $\pi$.  Therefore $\peak(\pi)$ is the total number of copies of the consecutive patterns $132$ and $231$ in $\pi$.   Likewise a \emph{valley} is a letter $\pi_i$ such that $\pi_{i-1} > \pi_i < \pi_{i+1}$.  If $\vall(\pi)$ is the number of valleys of $\pi$, then again we see that $\vall(\pi)$ is the total number of copies of the consecutive patterns $213$ and $312$ in $\pi$.  Therefore we see that $\peak(\pi)$ and $\vall(\pi)$ are ES-compatible statistics.

While exploring the distributions of the $\peak$ and $\vall$ statistics over sets $\Sav{n}{B}$, a few interesting confluences appeared: 

\begin{theorem}\label{thm:peaks}
 The following three distributions are equivalent:
   \begin{itemize}
    \item the distribution of peaks over $1\d2\d3$-avoiding permutations,
    \item the distribution of valleys over $1\d2\d3$-avoiding permutations, and
    \item the distribution of valleys over $1\d3\d2$-avoiding permutations.
  \end{itemize}
In the notation above, $F(\Sav{n}{1\d2\d3}, \peak, q) = F(\Sav{n}{1\d2\d3}, \vall, q) = F(\Sav{n}{1\d3\d2}, \vall, q)$.
\end{theorem}

\begin{remark}    In \cite{Dokos2012}, Dokos \etal started investigating statistic-Wilf-equivalence, where two pattern sets $B$ and $B'$ are $f$-Wilf-equivalent for permutation statistic $f$ if $F(\Sav{n}{B},f,q)=F(\Sav{n}{B'},f,q)$ for all $n$.  This equivalence is denoted $B\statwe{f} B'$.   In this notation, the second half of Theorem \ref{thm:peaks} can be written as
\begin{equation*}
 1\d2\d3 \statwe{\vall} 1\d3\d2.
\end{equation*}
\end{remark}

\begin{proof}
  The equivlence $F(\Sav{n}{1\d2\d3}, \peak, q) = F(\Sav{n}{1\d2\d3}, \vall, q)$ follows directly from symmetry.  Permutation $\pi$ avoids $1\d2\d3$ if and only if its reverse-complement $\pi^{rc}$ avoids $1\d2\d3$.  Next, the peaks in $\pi$ correspond to valleys in $\pi^{rc}$ and vice versa.  Thus it follows that $F(\Sav{n}{1\d2\d3}, \peak, q) = F(\Sav{n}{1\d2\d3},\vall, q)$.

 We next prove $F(\Sav{n}{1\d2\d3}, \vall, q) = F(\Sav{n}{1\d3\d2}, \vall, q)$.  We will begin by proving the following claim:  \textit{For any $\pi\in\Sav{n}{1\d2\d3}\cup\Sav{n}{1\d3\d2}$, $\pi_i$ is a valley of $\pi$ if and only if $i>1$, $\pi_i$ is a \ltr minimum of $\pi$ and $\pi_{i+1}$ is not an \ltr minimum of $\pi$.}
Suppose that $\pi_i$ is a valley but that there is some $\pi_j$ such that $\pi_j < \pi_i$ and $j< i$.  Then we see $\pi_j \pi_{i-1} \pi_i \oi 132$ and $\pi_j \pi_i \pi_{i+1} \oi 123$ and so $\pi$ would contain $1\d3\d2$ and $1\d2\d3$.  Thus if $\pi$ avoids $1\d3\d2$ or $1\d2\d3$ then $\pi_i$ must be a \ltr minimum.  Since $\pi_{i+1}>\pi_i$ we see that $\pi_{i+1}$ is not a \ltr minimum.  Conversely, suppose that $\pi_i$ is a \ltr minimum of $\pi$ for $i>1$ but $\pi_{i+1}$ is not.  Then $\pi_{i+1}>\pi_i$, since otherwise $\pi_{i+1}$ would be a \ltr minimum, and also $\pi_{i-1}>\pi_i$, since otherwise $\pi_i$ would not be an \ltr minimum.  Therefore $\pi_i$ must be a valley of $\pi$.  Hence the claim is proven.

The Simion-Schmidt bijection in \cite{Simion1985} provides a map $\Sav{n}{1\d3\d2} \to \Sav{n}{1\d2\d3}$ which preserves the  \ltr minima.  Therefore by the claim above, the same bijection also preserves the valleys of the permutations.  Thus we see that 
$F(\Sav{n}{1\d2\d3}, \vall, q) = F(\Sav{n}{1\d3\d2}, \vall, q)$.
\end{proof}

\begin{remark}
The proof above implies that $\vall$ is equally distributed over $\Sav{n}{1\d2\d3}$ and $\Sav{n}{1\d3\d2}$ even when restricting further to those permutations with a given set of \ltr minima.  Letting $\mathrm{LRmax}(\pi)$ be the set of indices which are \ltr minima for $\pi$, then for any $S\subseteq [n]$.
\begin{equation}
\sum_{\begin{subarray}{l}\pi\in\Sav{n}{1\d2\d3}\\ \mathrm{LRmax}(\pi)=S \end{subarray}} q^{\vall(\pi)} = 
\sum_{\begin{subarray}{l}\pi\in\Sav{n}{1\d3\d2}\\ \mathrm{LRmax}(\pi)=S \end{subarray}} q^{\vall(\pi)}
\end{equation}
\end{remark}

For completeness we will comment that $F(\Sav{n}{1\d3\d2},\peak,q)$ appears in OEIS as A091894, suggesting the following correspondence.  A Dyck path of semilength $n$ is a lattice path from $(0,0)$ to $(n,0)$ composed of steps $U=(1,1)$ and $D=(1,-1)$ which never goes below the $x$-axis.  We will write Dyck paths as words $U^{a_1}\,D^{b_1}\,U^{a_2}\,D^{b_2} \dotsm U^{a_k}\,D^{b_k}$.  

\begin{theorem}\label{thm:peaks2}
  The number of permutations in $\Sav{n}{1\d3\d2}$ with $k$ peaks equals the number of Dyck paths of semilength $n$ with $k$ occurrences of the subfactor $DDU$.
\end{theorem}

\begin{proof}
  Krattenthaler provides a bijection, $\Phi$, in \cite{Krattenthaler2001} from $\Sav{n}{1\d3\d2}$ to the set of Dyck paths of semilength $n$.  In that bijection, a permutation with $k$ valleys maps to a Dyck path with $k$ subfactors $DDU$.  The remainder of this proof outlines this bijection.

Suppose that $\pi\in\Sav{n}{1\d3\d2}$ with $\ltrmin(\pi)=k$.  Then let $j_1, j_2, \dotsc, j_k$ be the indices of the \ltr minima and let $m_i = \pi_{j_i}$ be the values of the \ltr minima for $1\leq i \leq k$.  Further, let $m_0 = n+1$ and $j_{k+1}=n+1$.  Then $\Phi(\pi)$ is the Dyck path $U^{a_1}\,D^{b_1}\,U^{a_2}\,D^{b_2} \dotsm U^{a_k}\,D^{b_k}$ where $a_i = m_{i-1} - m_i$ and $b_i = j_{i+1}-j_i$.  For $\Phi^{-1}$, observe that the \ltr minima of $\Phi^{-1}( U^{a_1}\,D^{b_1}\,U^{a_2}\,D^{b_2} \dotsm U^{a_k}\,D^{b_k})$ occur at indices $1, 1+b_1, 1+b_1+b_2, \dotsc, 1+b_1+b_2 + \dotsm b_{k-1}$ and have values $n+1 - m_1, n+1 - (m_1+m_2), \dotsc, n+1-(m_1 + m_2 + \dotsm m_k)$.  It is well-known that a $1\d3\d2$-avoiding permutation is uniquely determined by the indices and values of its \ltr minima, and thus $\Phi$ is bijective.

Observe that if $\ltrmin(\pi)=k$, then $\Phi(\pi) = U^{a_1}\,D^{b_1}\,U^{a_2}\,D^{b_2} \dotsm U^{a_k}\,D^{b_k}$ for $a_i>0$ and $b_i>0$.  Therefore each \ltr minima beyond $\pi_1$ corresponds to a $DU$ subfactor in $\Phi(\pi)$ .  Furthermore if adjecent letters $\pi_i$ and $\pi_{i+1}$ are both \ltr minima, then the corresponding string of $D$'s in the image is only a single $D$.  Therefore $DDU$ subfactors correspond to non-adjacent \ltr minima, which by the claim from Theorem \ref{thm:peaks} correspond to valleys in the permutation. 
\end{proof}

The distributions from Theorems \ref{thm:peaks} and \ref{thm:peaks2} are given in Tables \ref{tab:peaks1} and \ref{tab:peaks2}.  This data was generated by \texttt{Statter}.

\begin{table}[h]
 \begin{tabular}{c|l}
$n$ & $F(\Sav{n}{1\d2\d3}, \peak, q)$ \\
\hline
1 & $1$ \\
2 & $2$\\
3 & $3+2\,q$\\
4 & $4+10\,q$\\
5 & $5+32\,q+5\,{q}^{2}$\\
6 & $6+84\,q+42\,{q}^{2}$\\
7 & $7+198\,q+210\,{q}^{2}+14\,{q}^{3}$\\
8 & $8+438\,q+816\,{q}^{2}+168\,{q}^{3}$\\
9 & $9+932\,q+2727\,{q}^{2}+1152\,{q}^{3}+42\,{q}^{4}$\\
10 & $10+1936\,q+8250\,{q}^{2}+5940\,{q}^{3}+660\,{q}^{4} $\\
\end{tabular}
 \caption{The distributions from Theorem \ref{thm:peaks} for $1\leq n \leq 10$.  }  
 \label{tab:peaks1}
\end{table}

\begin{table}[h]
 \begin{tabular}{c|l}
$n$ & $F(\Sav{n}{1\d3\d2}, \peak, q)$ \\
\hline
1 & $1$ \\
2 & $2$\\
3 & $4+q$\\
4 & $8+6\,q$\\
5 & $16+24\,q+2\,{q}^{2}$\\
6 &  $32+80\,q+20\,{q}^{2}$\\
7 &  $64+240\,q+120\,{q}^{2}+5\,{q}^{3}$\\
8 & $128+672\,q+560\,{q}^{2}+70\,{q}^{3}$\\
9 &  $256+1792\,q+2240\,{q}^{2}+560\,{q}^{3}+14\,{q}^{4}$\\
10 &  $512+4608\,q+8064\,{q}^{2}+3360\,{q}^{3}+252\,{q}^{4}$\\
\end{tabular}
 \caption{The distributions of peaks over 1-3-2-avoiding permutations.  Corresponds to OEIS sequence A091894 \cite{OEIS}.}
 \label{tab:peaks2}
\end{table}

A complete classification of the classical patterns according to peak-Wilf-equivalence is forthcoming in \cite{Baxter2014peak}.

\subsection{The major index}\label{sec:majind}

The major index is defined by $\maj(\pi) := \sum_{i: \pi_{i}>\pi_{i+1}} i$.  As shown in \cite{Babson2000}, the major index can be decomposed into the sum of four vincular pattern functions:
\begin{equation}
 \maj(\pi) = (3\d21)(\pi) + (2\d31)(\pi) + (1\d32)(\pi) + (21)(\pi).
\end{equation}
where $(\sigma)(\pi)$ is the number of copies of $\sigma$ in permutaion $\pi$.
While these patterns do not conform to the structure of those in Theorem \ref{thm:tail}, their reverses do.  Thus we define
\begin{equation*}
 \rmaj(\pi) := \maj(\pi^{r}) = (12\d3)(\pi) + (13\d2)(\pi) + (23\d1)(\pi) + (12)(\pi).
\end{equation*}

Therefore Corollary \ref{cor:SchemeComputingMulti}, Theorem \ref{thm:tail} imply the following special case:

\begin{corollary}
If $B$ is a set of patterns such that $B^{r}:=\{\tau^{r}: \tau\in B\}$ admits a finite enumeration scheme, then $F( \Sav{n}{B},\maj,q)$ can be computed via enumeration scheme.
\end{corollary}

Distributions of the major index over avoidance sets $\Sav{n}{\tau}$ for classical patterns $\tau\in\S{3}$ are studied by Dokos \etal in \cite{Dokos2012}.  

To illustrate the method with a new example, consider the set of classical patterns 
$$B_k={\{2\d1\d3, 1\d2\d\dotsm\d(k-1)\d k\}}$$ for any $k\geq 2$.  A discussion of $B_3^c = \{2\d3\d1, 3\d2\d1\}$ appears in \cite{Dokos2012}, but we will consider the more general family.  It was shown in \cite{Vatter2008} that $B_k^r = \{ 3\d1\d2, k\d (k-1) \d \dotsm \d 2\d 1\}$ admits a finite enumeration scheme of depth 2:
\begin{equation}\label{eqn:Bkscheme}
\big\{ (\epsilon, \emptyset, \emptyset), (1, \emptyset, \emptyset), (12, \emptyset, \{1\}), (21, \{ \langle 0,1,0 \rangle, \langle k-2,0,0\rangle \}, \{1\}) \big\}
\end{equation}

The scheme in \eqref{eqn:Bkscheme} derives from general-purpose algorithms of \cite{Vatter2008}, and so is not necessarily optimal for a special case.  There are a few missed gap vector criteria which will simplify the resulting recurrence.
Observe that if $\pi$ is a permutation with $\pi_1\geq k$, then $2,3,\dotsc, k-1$ lie among the remaining letters.  Either these letters appear in decreasing order, in which case $\pi$ contains a $k\d (k-1) \d \dotsm \d 2\d 1$, or at least two of the letters appear in increasing order, in which case $\pi$ contains $3\d1\d2$.  Therefore $\langle k-1,0 \rangle$ is a gap vector for prefix pattern $1$, and in turn this implies that $\{\langle i,j,0 \rangle: i+j = k-1\}$ are gap vectors for prefix $12$.  Thus we arrive at the following enumeration scheme for $B_k^r$:

\begin{equation}\label{eqn:Bkscheme2}
\begin{split}
\big\{  (\epsilon, \emptyset, \emptyset), &(1, \{\langle k-1,0\rangle \}, \emptyset), (12, \{\langle k-1,0,0 \rangle, \langle k-2,1,0 \rangle, \dotsc, \langle 0,k-1,0 \rangle  \}, \{1\}),\\ & (21, \{ \langle 0,1,0 \rangle, \langle k-2,0,0\rangle \}, \{1\}) \big\} \\
\end{split}
\end{equation}

We will consider the following analogue of the classic Euler-Mahonian distribution, restricted to the $B_k$-avoiding permutations:
\begin{equation*}
 \begin{split}
G_n(p;w) &:=F( \Spt{n}{B_k}{p}{w}, \langle \maj, \des \rangle, \langle q, t \rangle) \\
                &= F( \Spt{n}{B_k^r}{p}{w}, \langle \rmaj, (12) \rangle, \langle q, t \rangle)
\end{split}
\end{equation*}

This scheme in \eqref{eqn:Bkscheme2} translates into the following recurrences for $n\geq 2$:

\begin{equation}
\begin{split}
  \sum_{\pi\in  \Sav{n}{B_k}} q^{\maj(\pi)}\,t^{\des(\pi)} &=   \sum_{\pi\in  \Sav{n}{B_k^r}} q^{\rmaj(\pi)}\,t^{(12)(\pi)}\\
    &=G_n(\epsilon; \epsilon) 
     =  \sum_{a=1}^{n} G_n(1, a) \\
      &=  \sum_{a=1}^{k-1} G_n(1, a) \textrm{  (by gap vector criteria)} \\
G_n(1,a) &= \sum_{b=1}^{a-1} G_n(21;ab) + \sum_{b=a+1}^{n} G_n(12;ab) \\
      &=  G_n(21;a(a-1)) + \sum_{b=a+1}^{k} G_n(12;ab) \textrm{  (by gap vector criteria)} \\
G_n(21;a(a-1) ) &= t^0\,q^0\,G_{n-1}(1;a-1) \\
G_n(12;ab ) &= t^1\,q^{n-1}\,G_{n-1}(1;b-1) \\
\end{split}
\end{equation}

The reader is left to verify the following values:
\begin{itemize}
  \item $\Delta_{\{1\}}^{(12)}(a(a-1),n) = 0$
  \item $\Delta_{\{1\}}^{\rmaj}(a(a-1),n) = 0$
  \item $\Delta_{\{1\}}^{(12)}(ab,n) = 1$  for $ab\oi 12$
  \item $\Delta_{\{1\}}^{\rmaj}(ab,n) = n-1$  for $ab\oi 12$
\end{itemize}

Combining the recurrence relations above yields the following recurrence for $G_n(1;a)$, the distribution of $(\maj, \des)$ over $B_k$-avoiding permutations:

\begin{equation}
G_n(1;a) = \begin{cases}
 t\,q^{n-1}\,\sum\limits_{b=1}^{k-1} G_{n-1}(1;b), & a=1 \\
 G_{n-1}(1;a-1) + t\,q^{n-1}\,\sum\limits_{b=a}^{k-1} G_{n-1}(1;b), & 2 \leq a \leq k-1 \\
 0, & a \geq k.
\end{cases}
\end{equation}

\section{Conclusion and Future Work}

The techniques above face the same limitations as enumeration schemes.  In short, the recurrences produced are often complicated and do not translate nicely into generating functions.  The methods discussed in Chapter 5 of \cite{BaxterThesis} make some progress toward converting schemes to generating functions, but cannot account for the full range of recurrences that schemes can produce.

Further, not all sets of vincular patterns $B$ admit a finite enumeration scheme, and there is no full characterization predicting whether a given $B$ will admit a finite scheme.  Data on how many sets $B$ do admit a small scheme are available in \cite{Baxter2012}.

We note that it should be possible to adapt the insertion encodings from \cite{Albert2005insertion, Vatter2009insenc} toward the purpose of computing $F(\Sav{n}{B}, f,q)$ for permutation statistics $f$ based on counting copies of consecutive patterns.  The insertion encoding offers two advantages over enumeration schemes: (1) the recurrences developed lead directly to generating functions, and (2) there are more [sets of] classical patterns which admit regular insertion encodings than finite enumeration schemes.  The current state of insertion encodings, however, cannot handle vincular patterns, however.  Such tools could a very helpful in classification of patterns under statistic-Wilf-equivalence for various statistics.

\bibliographystyle{plain}
\bibliography{EScompatibleStats-arxiv.bbl}

\end{document}